\documentclass{article}

\input{preamble}

\input{tcilatex}

\begin{document}
\title{Examples of renormalized SDEs}

%
%
\author{Y. Bruned, I. Chevyrev, and P. K. Friz}

\date{\today}


\maketitle

\begin{abstract}
We demonstrate two examples of stochastic processes whose lifts to geometric rough paths require a renormalisation procedure to obtain convergence in rough path topologies. Our first example involves a physical Brownian motion subject to a magnetic force which dominates over the friction forces in the small mass limit. Our second example involves a lead-lag process of discretised fractional Brownian motion with Hurst parameter $H \in (1/4,1/2)$, in which the stochastic area captures the quadratic variation of the process. In both examples, a renormalisation of the second iterated integral is needed to ensure convergence of the processes, and we comment on how this procedure mimics negative renormalisation arising in the study of singular SPDEs and regularity structures.
\end{abstract}


\section{Introduction}

In recent years, the theory of regularity structures~\cite{Hairer14} has been proposed to give meaning to a wide class of singular SPDEs. A central feature of the theory is the notion of renormalisation, specifically ``negative renormalisation''~\cite{BHZ16}, which is required to obtain convergence of random models to a meaningful limit. It is well-known that this procedure is inherit to the problem since naive approximations of such equations typically fail to converge (with a number of notable exceptions, including a special variant of gKPZ~\cite{Hairer16}). The same feature thus naturally appears in other solution theories which have been proposed to solve such equations, including the theories of paracontrolled distributions~\cite{GIP15} and the Wilsonian renormalisation group~\cite{Kupiainen16}.

Viewing regularity structures as a multidimensional generalisation of the theory of rough paths~\cite{Lyons98}, it is natural to ask how renormalisation manifests itself in the latter. As solution theories to singular S(P)DEs, both share the common feature that one must give meaning to often analytically ill-posed higher order terms (iterated integrals) of distributions, which is typically done through some stochastic means. However, a key difference in applications of rough paths to SDEs and rough SPDEs~\cite{FrizVictoir10} is that one can usually give meaning to such terms as limits of iterated integrals of mollifications of the irregular noise without the need of renormalisation.

The purpose of this note is to demonstrate situations in rough paths theory which fall outside this usual setting and for which renormalisation is a necessary feature. Specifically, we construct two stochastic processes whose lifts to geometric rough paths fail to converge without a renormalisation procedure akin to the one encountered in the theory of regularity structures.

The first is a physical Brownian motion subject to a magnetic force which dominates over the friction forces in the small mass limit. This example builds on the work~\cite{FrizGassiat15}, where a similar situation was considering with a constant magnetic field.

The second is a lead-lag process of a discretized path, which we take to be fractional Brownian motion with Hurst parameter $H \in (1/4,1/2]$. The stochastic area of this lead-lag process captures the quadratic variation of the discretized path, and thus, as one can expect, the second iterated integral fails to converge as the mesh of the discretisation goes to zero (unless $H=1/2$). This example is motivated from a similar Hoff process considered for semi-martinagles in~\cite{Flint16}.

In both examples we demonstrate an explicit renormalisation procedure of the second iterated integral under which the processes converge in rough path topologies. These (diverging) counter-terms serve precisely the same re-centring role encountered in regularity structures (for a direct comparison, consider the renormalisation of PAM~\cite{Hairer14,GIP15} where only one diverging term needs to be considered). In turn, rough differential equations driven by the renormalised and unrenormalised rough paths are related to one another by the addition of diverging terms, which again mimics the situation encountered in singular SPDEs. We refer to the upcoming work~\cite{Bruned16} for a much more detailed study of this relation.

\medskip

{\bf Acknowledgements.} P.K.F. is partially supported by the European Research Council through CoG-683166 and DFG research unit FOR2402. I.C., affiliated to TU Berlin when this project was commenced, was supported by DFG research unit FOR2402.

\section{Magnetic field blow-up}\label{subsec:magnetic}

Consider a physical Brownian motion in a magnetic field with dynamics given by
\[
m \ddot{x} = -A\dot{x} + B\dot{x} + \xi, \; \; x(t) \in \R^d, 
\]
where $A$ is a symmetric matrix with strictly positive spectrum (representing friction), $B$ is an anti-symmetric matrix (representing the Lorentz force due to a magnetic field), and $\xi$ is an $\R^d$-valued white noise in time. We shall consider the situation that $A$ is constant whereas $B$ is a function of the mass $m$.

We rewrite these dynamics as
\begin{align*}
dX_t &= \frac{1}{m} P_t dt, \; \; X_0 = 0, \\
dP_t &= -\frac{1}{m} M P_t dt + dW_t, \; \; P_0 = 0,
\end{align*}
where $M = A - B$, and we have chosen the starting point as zero simply for convenience. We furthermore introduce the parameter $\varepsilon^2 = m$ and write $X^\varepsilon_t, P^\varepsilon_t$, and $M^\varepsilon = A-B^\varepsilon$ to denote the dependence on $\varepsilon$.

We are interested in the convergence of the processes $P^\varepsilon$ and $M^\varepsilon X^\varepsilon$ in rough path topologies. Let $G^2(\R^d)$ and $\g^2(\R^d)$ denote the step-$2$ free nilpotent Lie group and Lie algebra respectively. Let us also write $\g^2(\R^d) = \R^d \oplus \g^{(2)}(\R^d)$ for the decomposition of $\g^2(\R^d)$ into the first and second levels, where we identify $\g^{(2)}(\R^d)$ with the space of anti-symmetric $d\times d$ matrices. 

For every $\varepsilon > 0$, define the matrix
\[
C^\varepsilon = \int_0^\infty e^{-M^\varepsilon s}e^{-(M^{\varepsilon})^* s}ds,
\]
and the element
\[
v^\varepsilon = -\frac{1}{2}(M^\varepsilon C^\varepsilon - C^\varepsilon(M^\varepsilon)^*) \in \g^{(2)}(\R^d).
\]


For any $v \in \g^{(2)}(\R^d)$, $p \in [1,3)$, and $p$-rough path $(Z_{s,t},\Z_{s,t}) \in G^2(\R^d)$ (where we ignore zeroth component $1$), we define the translated rough path $T_v(Z_{s,t}, \Z_{s,t})$ by
\begin{equation}\label{eq:TvDef}
T_v(Z_{s,t}, \Z_{s,t}) = (Z_{s,t}, \Z_{s,t} + (t-s)v^\varepsilon).
\end{equation}


Consider the $G^2(\R^d)$-valued processes
\begin{align*}
(P^\varepsilon_{s,t}, \Pbb^\varepsilon_{s,t}) &= \left(P^\varepsilon_{s,t}, \int_s^t P^\varepsilon_{s,r} \otimes \circ dP^\varepsilon_r\right), \\
(Z^\varepsilon_{s,t}, \Z^\varepsilon_{s,t}) &= \left(M^\varepsilon X^\varepsilon_{s,t}, \int_s^t M^\varepsilon X_{s,r} \otimes d(M^\varepsilon X^\varepsilon)_r \right),
\end{align*}
and the canonical lift of the Brownian motion $W$
\[
(W_{s,t}, \W_{s,t}) = \left(W_{s,t}, \int_s^t W_{s,r} \otimes \circ dW_r\right),
\]
where the integrals in the definition of $\Pbb^\varepsilon_{s,t}$ and $\W_{s,t}$ are in the Stratonovich sense.

The following proposition establishes the convergence of the ``renormalised'' paths $T_{v^\varepsilon}(P^\varepsilon_{s,t}, \Pbb^\varepsilon_{s,t})$ and $T_{v^\varepsilon}(Z^\varepsilon_{s,t}, \Z^\varepsilon_{s,t})$.

\begin{theorem}\label{thm:magneticConv}
Suppose that
\begin{equation}\label{eq:MBound}
\lim_{\varepsilon \rightarrow 0} |M^\varepsilon|\varepsilon^\kappa = 0 \; \textnormal{ for some } \kappa \in [0,1].
\end{equation}
Then for any $\alpha \in [0,1/2-\kappa/4)$ and $q < \infty$, it holds that $T_{v^\varepsilon}(P^\varepsilon, \Pbb^\varepsilon) \rightarrow (0,0)$ and $T_{v^\varepsilon}(Z^\varepsilon, \Z^\varepsilon) \rightarrow (W,\W)$ in $L^q$ and $\alpha$-H{\"o}lder topology as $\varepsilon \rightarrow 0$. More precisely, as $\varepsilon \rightarrow 0$, in $L^q$
\[
\sup_{s,t \in [0,T]} \frac{|P^\varepsilon_{s,t}|}{|t-s|^\alpha} + \sup_{s,t \in [0,T]} \frac{|\Pbb^\varepsilon_{s,t} + (t-s)v^\varepsilon|}{|t-s|^{2\alpha}} \rightarrow 0.
\]
and
\[
\sup_{s,t \in [0,T]} \frac{|Z^\varepsilon_{s,t} - W_{s,t}|}{|t-s|^\alpha} + \sup_{s,t \in [0,T]} \frac{|\Z^\varepsilon_{s,t} + (t-s)v^\varepsilon - \W_{s,t}|}{|t-s|^{2\alpha}} \rightarrow 0.
\]
\end{theorem}

The rest of the section is devoted to the proof of Theorem~\ref{thm:magneticConv} which builds on the proof of~\cite{FrizGassiat15} Theorem~1.

We set $Y^\varepsilon = P^\varepsilon/\varepsilon$ and obtain that
\begin{align*}
dY^\varepsilon_t &= -\varepsilon^2 M^\varepsilon Y^\varepsilon_t dt + \varepsilon^{-1}dW_t \\
dX^\varepsilon_t &= \varepsilon^{-1}Y^\varepsilon_t dt.
\end{align*}
For fixed $\varepsilon$, we introduce the Brownian motion $\tilde W^\varepsilon_\cdot = \varepsilon W_{\varepsilon^{-2} \cdot}$ and consider
\begin{align*}
d\tilde Y^\varepsilon_t &= -M^\varepsilon\tilde Y^\varepsilon_t dt + d\tilde W^\varepsilon_t.
\end{align*}
Observe that we have the pathwise equalities
\begin{equation}\label{eq:pathwiseEq}
Y^\varepsilon_\cdot = \tilde Y^\varepsilon_{\varepsilon^{-2}\cdot},
\end{equation}
and since $Y^\varepsilon_0 = 0$, we have
\begin{equation}\label{eq:YSol}
\tilde Y^\varepsilon_t = \int_0^t e^{-M^\varepsilon (t-s)} d\tilde W^\varepsilon_s.
\end{equation}


The dependence of $M^\varepsilon$ on $\varepsilon$ is, by construction, only though $B^\varepsilon$, the anti-symmetric part of $M^\varepsilon$. In particular, since the symmetric part $A$ stays constant and has strictly positive spectrum, it follows that for some $\lambda > 0$, $\Real(\sigma(M^\varepsilon)) \subset (\lambda,\infty)$ for all $\varepsilon > 0$. In particular,
\begin{equation}\label{eq:expBound}
\sup_{\tau > 0} \sup_{\varepsilon > 0} \frac{|e^{-\tau M^\varepsilon}|}{e^{-\lambda \tau}} < \infty.
\end{equation}
We see then that
\[
\sup_{\varepsilon > 0} |C^\varepsilon| < \infty
\]
and
\begin{equation}\label{YtildeL2}
\sup_{\varepsilon > 0} \sup_{0 \leq t < \infty} \EEE{|\tilde Y^\varepsilon_t|^2} < \infty.
\end{equation}

\begin{lemma}\label{lem:YBounds}
There exists $C_1 > 0$ such that for all $\varepsilon \in (0,1]$ and $s,t \in [0,T]$
\[
\EEE{\left| Y^{\varepsilon}_{s,t} \right|^2}^{1/2} \leq C_1 \min\{\varepsilon^{-1}|t-s|^{1/2}, 1\}
\]
and
\[
\EEE{\left| \int_s^t Y^{\varepsilon}_r \otimes Y^{\varepsilon}_r dr - (t-s)C^\varepsilon \right|^2}^{1/2} \leq C_1 \min\{\varepsilon |t-s|^{1/2}, |t-s|\}.
\]
\end{lemma}

\begin{proof}
The first inequality is clear from~\eqref{eq:pathwiseEq},~\eqref{eq:YSol} and~\eqref{YtildeL2}. For the second, from~\eqref{eq:YSol}, we see that for every $r > 0$, $\tilde Y^{\varepsilon}_r$ has distribution $\NN(0, C^\varepsilon_r)$ where
\[
C^\varepsilon_r = \int_0^r e^{-M^\varepsilon (r-u)} e^{-(M^\varepsilon)^*(r-u)} du = \int_0^r e^{-M^\varepsilon u} e^{-(M^\varepsilon)^* u} du.
\]
Hence $Y^\varepsilon_r = \tilde Y_{\varepsilon^{-2}r}$ has distribution $\NN(0, C^\varepsilon_{\varepsilon^{-2} r})$. Thus
\[
\EEE{\int_s^t Y^{\varepsilon}_r \otimes Y^{\varepsilon}_r dr} = \int_s^t C^\varepsilon_{\varepsilon^{-2} r} dr = \int_s^t \int_0^{\varepsilon^{-2} r} e^{-M^\varepsilon u} e^{-(M^\varepsilon)^* u} du dr =: \mu_{s,t}^\varepsilon.
\]
Observe that from~\eqref{eq:expBound}
\begin{align*}
|\mu_{s,t}^\varepsilon - (t-s) C^\varepsilon| &\leq \int_s^t \int_{\varepsilon^{-2} r}^\infty |e^{-M^\varepsilon u} e^{-(M^\varepsilon)^* u}| du dr \\
&\leq C_2\int_s^t \int_{\varepsilon^{-2}r}^\infty  e^{-2\lambda u} dudr \\
&\leq C_3 \int_s^t e^{-2\lambda \varepsilon^{-2}r}dr \\
&\leq C_4 \min\{\varepsilon^2, |t-s|\} \\
&\leq C_4 \min\{\varepsilon |t-s|^{1/2}, |t-s|\}.
\end{align*}
We now claim that
\[
\EEE{\left| \int_s^t Y^{\varepsilon}_r \otimes Y^{\varepsilon}_r dr - \mu_{s,t}^\varepsilon \right|^2} \leq C_5 \min\{\varepsilon^2 |t-s|, |t-s|^2\},
\]
from which the conclusion follows. Indeed, by Fubini and Wick's formula
\begin{align*}
\EEE{\left(\int_s^t Y^{\varepsilon,i}_r Y^{\varepsilon,j}_r dr\right)^2}
&= \int_{[s,t]^2} \EEE{Y^{\varepsilon,i}_r Y^{\varepsilon,j}_r Y^{\varepsilon,i}_u Y^{\varepsilon,j}_u} dr du \\
&= \int_{[s,t]^2} \EEE{Y^{\varepsilon,i}_r Y^{\varepsilon,j}_r}\EEE{Y^{\varepsilon,i}_u Y^{\varepsilon,j}_u} dr du \\
&+ \int_{[s,t]^2}  \EEE{Y^{\varepsilon,i}_r Y^{\varepsilon,i}_u}\EEE{Y^{\varepsilon,j}_r Y^{\varepsilon,j}_u} dr du \\
&+ \int_{[s,t]^2}  \EEE{Y^{\varepsilon,i}_r Y^{\varepsilon,j}_u}\EEE{Y^{\varepsilon,j}_r Y^{\varepsilon,i}_u} dr du \\
&\leq (\mu_{i,j}^\varepsilon)_{s,t}^2 + 4\int_{[s,t]^2} \left|\EEE{Y^{\varepsilon}_u \otimes Y^{\varepsilon}_r}\right|^2 \mathbf{1}\{r \leq u\} dr du.
\end{align*}

Observe that for $r \leq u$
\[
\EEE{Y^\varepsilon_u \mid Y^\varepsilon_r} = e^{-\varepsilon^{-2} M^\varepsilon(u-r)} Y_r^\varepsilon.
\]
and so
\[
|\EEE{Y^\varepsilon_u\otimes Y^\varepsilon_r}|^2 \mathbf{1}\{r \leq u\}  \leq C_6 e^{-\varepsilon^{-2}2\lambda(u-r)} |C^\varepsilon_{\varepsilon^{-2}r}| \leq C_7 e^{-\varepsilon^{-2}2\lambda(u-r)}.
\]
Thus
\begin{align*}
\EEE{\left(\int_s^t Y^{\varepsilon,i}_r Y^{\varepsilon,j}_r dr - (\mu_{i,j}^\varepsilon)_{s,t}\right)^2}
&= \EEE{\left(\int_s^t Y^{\varepsilon,i}_r Y^{\varepsilon,j}_r dr\right)^2} - (\mu_{i,j}^\varepsilon)_{s,t}^2 \\
&\leq C_8 \int_s^t \int_r^t e^{-\varepsilon^{-2}2\lambda(u-r)} du dr \\
&\leq C_9 \min\{\varepsilon^2 |t-s|, |t-s|^2\}
\end{align*}
as claimed.
\end{proof}

\begin{lemma}\label{lem:PBounds}
There exists $C_{10} > 0$ such that for all $\varepsilon \in (0,1]$ and $s,t \in [0,T]$
\[
\norm{P^\varepsilon_{s,t}}_{L^2} \leq C_{10} \min\{\varepsilon, |t-s|^{1/2}\}
\]
and
\[
\norm{\Pbb^\varepsilon_{s,t} + (t-s)v^\varepsilon}_{L^2} \leq C_{10} |M^\varepsilon| \min\{\varepsilon |t-s|^{1/2}, |t-s|\}
\]
\end{lemma}

\begin{proof}
The first inequality is immediate from Lemma~\ref{lem:YBounds}. For the second, we have
\begin{align*}
\Pbb^\varepsilon_{s,t} &= \varepsilon^2 \int_s^t Y^\varepsilon_{s,r} \otimes \circ dY^\varepsilon_r \\
&= -\int_s^t Y^\varepsilon_{s,r} \otimes M^\varepsilon Y^\varepsilon_r dr + \varepsilon\int_s^t Y^\varepsilon_{s,r}\otimes dW_r + \frac{1}{2}(t-s)I.
\end{align*}
Since $Y^\varepsilon_{s,r} \otimes M^\varepsilon Y^\varepsilon_r = (Y^\varepsilon_{s,r} \otimes Y^\varepsilon_r)(M^\varepsilon)^*$ and we can directly verify that $v^\varepsilon = C^\varepsilon(M^\varepsilon)^* - \frac{1}{2}I$, we have
\[
\Pbb^\varepsilon_{s,t} + (t-s)v^\varepsilon = -\left(\int_s^t Y^\varepsilon_{s,r} \otimes Y^\varepsilon_r dr - (t-s)C^\varepsilon\right) (M^\varepsilon)^* + \varepsilon \int_s^t Y^\varepsilon_{s,r} \otimes dW_r.
\]
From Lemma~\ref{lem:YBounds}, we see that
\[
\norm{\varepsilon \int_s^t Y^\varepsilon_{s,r} \otimes dW_r}_{L^2} \leq C_{11}\min\{\varepsilon|t-s|^{1/2},|t-s|\}.
\]
Furthermore, by Fubini and Wick's formula, we can readily show
\[
\norm{\int_s^t Y^\varepsilon_s \otimes Y^\varepsilon_r dr}_{L^2} \leq C_{12} \min\{\varepsilon |t-s|^{1/2}, |t-s|\}.
\]
It now follows from Lemma~\ref{lem:YBounds} that
\[
\norm{\Pbb^\varepsilon_{s,t} + (t-s)v^\varepsilon}_{L^2} \leq  C_{13} |M^\varepsilon| \min\{\varepsilon |t-s|^{1/2}, |t-s|\}.
\]
\end{proof}

\begin{lemma}\label{lem:ZBounds}
There exists $C_{14} > 0$ such that for all $\varepsilon \in (0,1]$ and $s,t \in [0,T]$
\[
\EEE{|Z^\varepsilon_{s,t} - W_{s,t}|^2}^{1/2} \leq C_{14} \min\{\varepsilon, |t-s|^{1/2}\}
\]
and
\[
\EEE{\left|\Z^\varepsilon_{s,t} + (t-s)v^\varepsilon - \W_{s,t}\right|^2}^{1/2} \leq C_{14} |M^\varepsilon| \min\{\varepsilon |t-s|^{1/2}, |t-s|\}.
\]
\end{lemma}

\begin{proof}
The first inequality follows from $Z^\varepsilon_{s,t} = W_{s,t} - \varepsilon Y^\varepsilon_{s,t}$ and Lemma~\ref{lem:YBounds}. For the second, we have
\begin{align*}
\int_s^t Z^\varepsilon_{s,r} \otimes dZ^\varepsilon_r
&= \int_s^t Z^\varepsilon_{s,r} \otimes dW_r - \varepsilon\left( \int_s^t Z^\varepsilon_{r} \otimes dY^\varepsilon_r - Z^\varepsilon_s\otimes Y^\varepsilon_{s,t}\right) \\
&= \int_s^t Z^\varepsilon_{s,r} \otimes dW_r - \varepsilon \left( Z^\varepsilon_t\otimes Y^\varepsilon_t - \int_s^t dZ_r \otimes Y^\varepsilon_r - Z^\varepsilon_s\otimes Y^\varepsilon_t\right) \\
&= \int_s^t Z^\varepsilon_{s,r} \otimes dW_r - \varepsilon Z_{s,t}^\varepsilon\otimes Y_t^\varepsilon + \int_s^t M^\varepsilon Y^\varepsilon_r \otimes Y^\varepsilon_r dr.
\end{align*}
We see that
\[
\norm{\int_s^t Z^\varepsilon_{s,r} \otimes dW_r - \int_s^t W_{s,r} \otimes dW_r}^2_{L^2} \leq C_{15}\min\{\varepsilon^2|t-s|, |t-s|^2\}.
\]
Furthermore, by Fubini and Wick's formula, we can readily show
\[
\norm{\varepsilon Z^\varepsilon_{s,t}\otimes Y^\varepsilon_t}^2_{L^2} = \norm{\int_s^t M^\varepsilon Y^\varepsilon_{r}\otimes Y^\varepsilon_t dr}^2_{L^2} \leq C_{16}|M^\varepsilon| \min\{\varepsilon^2 |t-s|, |t-s|^2\}.
\]
Finally, by Lemma~\ref{lem:YBounds}
\[
\norm{\int_s^t M^\varepsilon Y^\varepsilon_r \otimes Y^\varepsilon_r dr - (t-s)M^\varepsilon C^\varepsilon|}_{L^2} \leq C_{17}|M^\varepsilon| \min \{\varepsilon |t-s|^{1/2}, |t-s|\}.
\]
It follows that
\[
\norm{\Z^\varepsilon_{s,r} - \W_{s,t} - (t-s)(M^\varepsilon C^\varepsilon - \frac{1}{2}I)}_{L^2} \leq C_{18}|M^\varepsilon| \min \{\varepsilon |t-s|^{1/2}, |t-s|\}.
\]
We can directly verify $v^\varepsilon = -M^\varepsilon C^\varepsilon + \frac{1}{2}I$, from which the conclusion follows.
\end{proof}

\begin{proof}[Proof of Theorem~\ref{thm:magneticConv}]
Observe that condition~\eqref{eq:MBound} implies that
\[
\lim_{\varepsilon \rightarrow 0}|M^\varepsilon|\varepsilon = 0.
\]
From Lemmas~\ref{lem:PBounds} and~\ref{lem:ZBounds}, along with Gaussian chaos, we thus obtain the pointwise convergence as $\varepsilon \rightarrow 0$ for any $q < \infty$ and $s,t \in [0,T]$ in $L^q$
\[
|P^\varepsilon_{s,t}| + |\Pbb^\varepsilon_{s,t} + (t-s)v^\varepsilon|^{1/2} \rightarrow 0
\]
and
\[
|Z^\varepsilon_{s,t} - W_{s,t}| + |\Z^\varepsilon_{s,t} + (t-s)v^\varepsilon - \W_{s,t}|^{1/2} \rightarrow 0.
\]
Furthermore, since $\min\{\varepsilon |t-s|^{1/2}, |t-s|\} \leq \varepsilon^{\kappa}|t-s|^{1-\kappa/2}$ for all $\kappa \in [0,1]$, condition~\eqref{eq:MBound}, Lemmas~\ref{lem:PBounds} and~\ref{lem:ZBounds}, and Gaussian chaos imply that for any $q < \infty$ there exists $C_q > 0$ such that for all $s,t \in [0,T]$ and
\begin{align*}
\sup_{\varepsilon \in (0,1]} \EEE{|P^\varepsilon_{s,t}|^q} &\leq C_q|t-s|^{q/2}, \\
\sup_{\varepsilon \in (0,1]} \EEE{|Z^\varepsilon_{s,t} - W_{s,t}|^q} &\leq C_q|t-s|^{q/2}
\end{align*}
and
\begin{align*}
\sup_{\varepsilon \in (0,1]} \EEE{|\Pbb^\varepsilon_{s,t} + (t-s)v^\varepsilon|^q} &\leq C_q|t-s|^{q(1-\kappa/2)}, \\
\sup_{\varepsilon \in (0,1]} \EEE{|\Z^\varepsilon_{s,t} + (t-s)v^\varepsilon - \W_{s,t}|^q} &\leq C_q|t-s|^{q(1-\kappa/2)}.
\end{align*}
Applying Theorem~A.13 of~\cite{FrizVictoir10} completes the proof.
\end{proof}

\section{Rough lead-lag process}\label{subsec:Hoff}

Consider a path $X : [0,1] \mapsto \R^d$. Let $n \geq 1$ be an integer and write for brevity $X^n_i = X_{i/n}$. Consider the piecewise linear path $\tilde X^n : [0,1] \mapsto \R^{2d}$ defined by
\begin{align*}
\tilde X^n_{2i/2n} &= (X^n_i, X^n_i), \\
\tilde X^n_{(2i+1)/2n} &= (X^n_i, X^n_{i+1}),
\end{align*}
and linear on the intervals $\left[\frac{2i}{2n}, \frac{2i+1}{2n}\right]$ and $\left[\frac{2i+1}{2n}, \frac{2i+2}{2n}\right]$ for all $i = 0,\ldots, n-1$. Note that this is a variant of the Hoff process considered in~\cite{Flint16}.

Denote by $\tilde \Xbf^n_{s,t} = \exp(\tilde X^n_{s,t} + \Abb^n_{s,t})$ the level-$2$ lift of $\tilde X^n$, where $\Abb^n_{s,t}$ is the $(2d) \times (2d)$ anti-symmetric L{\'e}vy area matrix given by
\[
\Abb^n_{s,t}
= \frac{1}{2}\left(\int_s^t \tilde X^n_{s,r} \otimes d\tilde X^n_r - \int_s^t \tilde X^n_{s,r} \otimes d\tilde X^n_r\right).
\]

Let $H \in (0,1)$ and consider a fractional Brownian motion $B^H$ with covariance $R(s,t) = \frac{1}{2}(t^{2H} + s^{2H} - |t-s|^{2H})$. Let $X : [0,1] \mapsto \R^d$ be $d$ independent copies of $B^H$.

Recall the definition of $T_v$ from~\eqref{eq:TvDef}. We are interested in the convergence in rough path topologies of $T_{\tilde v^n}(\tilde \Xbf^n)$ where $\tilde v^n \in \g^{(2)}(\R^{2d})$ is appropriately chosen. Define the (diagonal) $d\times d$ matrix
\[
v^n = \frac{1}{2}\EEE{\sum_{k=0}^{n-1} (X^n_{k+1}-X^n_k) \otimes (X^n_{k+1}-X^n_k)} = \frac{n^{1-2H}}{2} I,
\]
and the anti-symmetric $(2d) \times (2d)$ matrix
\[
\tilde v^n = \left( \begin{array}{cc} 0 & -v^n  \\ v^n  & 0 \end{array} \right) \in \g^{(2)}(\R^{2d}).
\]
Finally, consider the path $\tilde X = (X,X) : [0,1] \mapsto \R^{2d}$, its canonically defined L{\'e}vy area $\Abb$ (which exists for $1/4 < H \leq 1$), and its level-$2$ lift $\tilde \Xbf = \exp(\tilde X + \Abb)$. The following is the main result of this subsection.

\begin{theorem}\label{thm:HoffConv}
Suppose $1/4 < H \leq 1/2$. Then for all $\alpha \in [0, H)$ and $q < \infty$, it holds that $T_{\tilde v^n}(\tilde \Xbf^n) \rightarrow \tilde \Xbf$ in $L^q$ and $\alpha$-H{\"o}lder topology. More precisely, as $n \rightarrow \infty$, in $L^q$
\[
\sup_{s,t \in [0,T]} \frac{|\tilde X^n_{s,t} - \tilde X_{s,t}|}{|t-s|^\alpha} + \sup_{s,t \in [0,T]} \frac{|\Abb^n_{s,t} + (t-s)\tilde v^n - \Abb_{s,t}|}{|t-s|^{2\alpha}} \rightarrow 0.
\]
\end{theorem}

The rest of the section is devoted to the proof of Theorem~\ref{thm:HoffConv}. We first state two lemmas which are purely deterministic.

Let $Y^n : [0,1] \mapsto \R^{d}$ be the piecewise linear interpolation of $X$ over the partition $\left(0, \frac{1}{n}, \ldots, \frac{n-1}{n}, 1\right)$, let $\tilde Y^n = (Y^n,Y^n) : [0,1] \mapsto \R^{2d}$,
and let $\Ybb^n$ be the L{\'e}vy area of $\tilde Y^n$.

\begin{lemma}\label{lem:XYDiff}
Let $s \in [\frac{m}{n},\frac{m+1}{n}]$ and $t \in [\frac{k}{n}, \frac{k+1}{n}]$ with $s < t$, and define
\begin{align*}
\Delta_1 &= n\left(\frac{m+1}{n} \wedge t - s \right) |X^n_{m+1} - X^n_m|, \\
\Delta_2 &= |X^n_k - X^n_{m+1}| \; \textnormal{ if $k > m$}, \; \; 0 \textnormal{ if  $k= m$} \\
\Delta_3 &= n\left(t-\frac{k}{n}\vee s\right)|X^n_{k+1} - X^n_k|.
\end{align*}
There exists a constant $C_1 > 0$ such that for all $n \geq 1$ and $0 \leq s < t \leq 1$, it holds that
\[
|\tilde X^n_{s,t} - \tilde Y^n_{s,t}| \leq C_1\left( \Delta_1 + \Delta_3 \right).
\]
and, if $k > m$, $|\Abb^n_{s,t} - \Abb^n_{\frac{m+1}{n}, \frac{k}{n}}|$ and $|\Ybb^n_{s,t} - \Ybb^n_{\frac{m+1}{n}, \frac{k}{n}}|$ are bounded above by
\[
C_1 \left(\Delta_1^2 + (\Delta_1+ \Delta_2 + \Delta_3)\Delta_3 + \Delta_1 \Delta_2 \right).
\]
\end{lemma}

\begin{proof}
Direct calculation and triangle inequality.
\end{proof}

The second part of the above lemma essentially allows us to work over the partition $\left(0, \frac{1}{n}, \ldots, \frac{n-1}{n}, 1\right)$, on which computations are easier.

\begin{lemma}\label{lem:partitionPoints}
Suppose $0 \leq m \leq k \leq n$.

1) For all pairs $1 \leq i,j \leq d$ and $d+1 \leq i,j \leq 2d$
\[
\left(\Abb^n_{\frac{m}{n}, \frac{k}{n}}\right)^{i,j} = \left(\Ybb^n_{\frac{m}{n}, \frac{k}{n}}\right)^{i,j}.
\]

2) For all $1 \leq i \leq d < j \leq 2d$
\[
\left(\Abb^n_{\frac{m}{n}, \frac{k}{n}}\right)^{i,j} = \left(\Ybb^n_{\frac{m}{n}, \frac{k}{n}}\right)^{i,j} - \frac{1}{2} \sum_{r=m}^{k-1} (X^{n,i}_{r+1}-X^{n,i}_r)(X^{n,j}_{r+1}-X^{n,j}_r)
\]
\end{lemma}

\begin{proof}
Denote $\tilde X^n = (M^n, N^n)$, so that $M^n$ is the lag component, and $N^n$ is the lead. The first equality is clear since $M^n$ and $N^n$ are simply reparametrisations of $Y^n$ over the interval $[\frac{m}{n},\frac{k}{n}]$. For the second, observe that
\[
\int_{m/n}^{k/n} M^{n,i}_{m/n, r} dN^{n,j}_r = \sum_{r=m}^{k-1} (X^{n,i}_{r}-X^{n,i}_{m}) (X^{n,j}_{r+1} - X^{n,j}_{r})
\]
and
\[
\int_{m/n}^{k/n} N^{n,j}_{m/n, r} dM^{n,i}_r = \sum_{r=m}^{k-1} (X^{n,j}_{r+1}-X^{n,j}_{m}) (X^{n,i}_{r+1} - X^{n,i}_{r}).
\]
Remark now that the signature of $Y^n$ over $[\frac{m}{n},\frac{k}{n}]$ is
\[
e^{X_{m+1} - X_{m}} \ldots e^{X_k - X_{k-1}},
\]
so that a calculation with the CBH formula gives
\begin{align*}
\left(\Ybb^n_{m/n,k/n}\right)^{i,j}
&= \frac{1}{2}\sum_{r=m}^{k-1} (X^{n,i}_r-X^{n,i}_m)(X^{n,j}_{r+1}-X^{n,j}_r) - (X^{n,j}_r-X^{n,j}_m)(X^{n,i}_{r+1}-X^{n,i}_r).
\end{align*}
Using the fact that
\begin{align*}
\left(\Abb^n_{m/n,k/n}\right)^{i,j} &= \frac{1}{2}\left( \int_{m/n}^{k/n} M^{n,i}_{m/n, r} dN^{n,j}_r - \int_{m/n}^{k/n} N^{n,j}_{m/n, r} dM^{n,i}_r \right),
\end{align*}
the conclusion readily follows.
\end{proof}

We now return to the specific case that $X : [0,1] \mapsto \R^d$ is given by $d$-independent copies of a fractional Brownian motion with Hurst parameter $H \in (0,1)$. In particular, this implies that for all $s,t,a,b \in [0,1]$
\begin{equation}\label{eq:covariance}
\EEE{(X^i_t-X^i_s)(X^j_b-X^j_a)} = \delta_{i,j} \frac{1}{2}(|t-a|^{2H} + |s-b|^{2H} - |t-b|^{2H} - |s-a|^{2H}).
\end{equation}
Consider the $(2d) \times (2d)$ anti-symmetric matrix
\[
\Pbb^n_{s,t} = \Abb^n_{s,t} - \Ybb^n_{s,t} + (t-s)\tilde v^n.
\]

\begin{lemma}\label{lem:PControlPartition}
There exists $C_2 > 0$ such that for all $H \leq 1/2$, $n \geq 1$, and $0 \leq m \leq k \leq n$
\[
\norm{\Pbb^n_{m/n,k/n}}_{L^2} \leq C_2 \frac{(k-m)^{1/2}}{n^{2H}}.
\]
\end{lemma}

\begin{proof}
Denote $K = k-m$. By part (2) of Lemma~\ref{lem:partitionPoints}, we have
\[
\norms{\Pbb^n_{m/n,k/n}} \leq \sum_{i,j=1}^d \norms{\sum_{r=m}^{k-1} (X^{n,i}_{r+1} - X^{n,i}_r)(X^{n,j}_{r+1}- X^{n,j}_r) - \frac{K}{n}v^n_{i,j}}.
\]
Observe moreover that
\[
\frac{K}{n}v^n_{i,j} = \EEE{\sum_{r=m}^{k-1}(X^{n,i}_{r+1} - X^{n,i}_r)(X^{n,j}_{r+1}- X^{n,j}_r)} = \delta_{i,j} K n^{-2H},
\]
and that for all $r,\ell \in \{0,\ldots, n-1\}$
\[
\EEE{(X^{n,i}_{r+1} - X^{n,i}_r)(X^{n,i}_{\ell+1} - X^{n,i}_\ell)} = \frac{n^{-2H}}{2}(|r-\ell+1|^{2H} + |r-\ell-1)|^{2H} - 2|r-\ell|^{2H}).
\]
Then for all $i \neq j$, by independence of the components of $X$,
\begin{align*}
& \EEE{\left(\sum_{r=m}^{k-1}  (X^{n,i}_{r+1} - X^{n,i}_r)(X^{n,j}_{r+1}- X^{n,j}_r) \right)^2}
\\ &= \EEE{\sum_{r=m}^{k-1} \sum_{\ell=m}^{k-1} X^i_{r,r+1}X^i_{\ell,\ell+1} X^j_{r,r+1} X^j_{\ell,\ell+1}} \\
&= \sum_{r=m}^{k-1} \sum_{\ell=m}^{k-1} \EEE{(X^{n,i}_{r+1} - X^{n,i}_r)(X^{n,i}_{\ell+1} - X^{n,i}_{\ell})}^2 \\
&= \sum_{r=m}^{k-1} \sum_{\ell=m}^{k-1} \frac{n^{-4H}}{4}\left(|r-\ell+1|^{2H} + |r-\ell-1|^{2H} - 2|r-\ell|^{2H} \right)^2 \\
&= \frac{n^{-4H}}{4} \sum_{x=-K+1}^{K-1} (K-|x|)(|x+1|^{2H} + |x-1|^{2H} - 2x^{2H})^2 \\
&=: \psi(n,K).
\end{align*}
Likewise for $i=j$, by Wick's formula,
\begin{align*}
\EEE{\left(\sum_{r=m}^{k-1} (X^{n,i}_{r+1} - X^{n,i}_r)^2 \right)^2}
&= \sum_{r=m}^{k-1} \sum_{\ell=m}^{k-1} \EEE{(X^{n,i}_{r+1} - X^{n,i}_{r})^2} \EEE{(X^{n,i}_{\ell+1} - X^{n,i}_{\ell})^2} \\
&+ 2\EEE{(X^{n,i}_{r+1} - X^{n,i}_{r})(X^{n,i}_{\ell+1} - X^{n,i}_{\ell})}^2  \\
&= K^2 n^{-4H} + 2\psi(n,K) \\
&= \left(\frac{K}{n}v^n_{i,i}\right)^2 + 2\psi(n,K).
\end{align*}
It hence follows that
\[
\norm{\Pbb^n_{m/n, k/n}}_{L^2}^2 \leq C_3\psi(n,K).
\]
The conclusion now follows since one can readily show that there exists $C_4 > 0$ such that for all $H \leq 1/2$, $n \geq 1$, and $0 \leq K \leq n$
\[
\psi(n,K) \leq C_4K n^{-4H}
\]
(in fact the inequality holds for all $H < 3/4$, though with a constant in general depending on $H$).
\end{proof}

\begin{lemma}\label{lem:YPBounds}
There exists $C_5 > 0$ such that for all $H \in \left(\frac{1}{4},\frac{1}{2}\right]$, $n \geq 1$ and $0 \leq s < t \leq 1$
\[
\norm{\Pbb^n_{s,t}}_{L^2} \leq C_5|t-s|^{2H}.
\]
\end{lemma}

\begin{proof}
Suppose $s \in [\frac{m}{n},\frac{m+1}{n}]$ and $t \in [\frac{k}{n}, \frac{k+1}{n}]$.
If $m=k$, then $\Abb^n_{s,t} = \Ybb^n_{s,t} = 0$ and $|t-s| < n^{-1}$, so that
\[
|\Pbb^n_{s,t}| = \left(t- s\right)|\tilde v^n| \leq |t-s|^{2H}.
\]
For the case $k > m$, following the notation of Lemma~\ref{lem:XYDiff}, note that $\EEE{\Delta_1^2}$ and $\EEE{\Delta_3^2}$ are bounded above by
\[
n^{-2H} \min\{n^2|t-s|^2,1\}.
\]
It readily follows that for all $\ell \in \{1,2,3\}$
\[
\EEE{\Delta^2_\ell} \leq |t-s|^{2H}.
\]
Moreover, we have
\[
\left(t-\frac{k}{n} + \frac{m+1}{n} - s\right)|\tilde v^n| \leq \min \{|t-s|, n^{-1} \} n^{1-2H} \leq |t-s|^{2H}.
\]
Hence Lemma~\ref{lem:XYDiff} implies that
\begin{multline*}
|\Pbb^n_{s,t} - \Pbb^n_{(m+1)/n, k/n}| \leq 2C_1\left(\Delta_1^2 + (\Delta_1+ \Delta_2 + \Delta_3)\Delta_3 + \Delta_1 \Delta_2 \right) \\ + \left(t-\frac{k}{n} + \frac{m+1}{n} - s\right)|\tilde v^n|,
\end{multline*}
and so by Gaussian chaos
\[
\norm{\Pbb^n_{s,t} - \Pbb^n_{(m+1)/n, k/n}}_{L^2} \leq C_6|t-s|^{2H}.
\]
The conclusion now follows from Lemma~\ref{lem:PControlPartition} since $(k-m-1)^{1/2}n^{-2H} \leq |t-s|^{2H}$ for all $H \geq 1/4$.
\end{proof}

\begin{proof}[Proof of Theorem~\ref{thm:HoffConv}]
Let $0 \leq s < t \leq 1$. We observe that as $n \rightarrow \infty$, it readily follows from Lemmas~\ref{lem:XYDiff} and~\ref{lem:PControlPartition} that in $L^q$
\[
|\tilde X^n_{s,t} - \tilde Y^n_{s,t}| \rightarrow 0
\]
and
\[
|\Abb^n_{s,t} + (t-s)\tilde v^n - \Ybb^n_{s,t}| \rightarrow 0.
\]
Furthermore, by Gaussian chaos, Lemma~\ref{lem:XYDiff} implies
\[
\sup_{n \geq 1} \EEE{|\tilde X^n_{s,t} - \tilde Y^n_{s,t}|^q} \leq C_q |t-s|^{qH},
\]
while Lemma~\ref{lem:YPBounds} implies
\[
\sup_{n \geq 1} \EEE{|\Abb^n_{s,t} + (t-s)\tilde v^n - \Ybb^n_{s,t}|^q} \leq C_q |t-s|^{2qH}.
\]
Applying Theorem~A.13 of~\cite{FrizVictoir10}, and the fact that $\exp(\tilde Y^n + \Ybb^n) \rightarrow \tilde \Xbf$ in $\alpha$-H{\"o}lder topology in $L^q$ (\cite{FrizVictoir10} Theorem~15.42), completes the proof.
\end{proof}

\bibliographystyle{plain}
\bibliography{AllRefs}

\end{document}